\documentclass{article}

\usepackage{amsmath,amssymb,amsthm}

\newtheorem{theorem}{Theorem}
\newtheorem{corollary}{Corollary}

\theoremstyle{remark}
\newtheorem{remark}{Remark}

\theoremstyle{definition}
\newtheorem{definition}{Definition}
\newtheorem{example}{Example}

\newcommand{\T}{\mathbb{T}}
\newcommand{\R}{\mathbb{R}}
\newcommand{\Z}{\mathbb{Z}}
\newcommand{\rar}{\mbox{$\rightarrow$}}

\begin{document}

\title{Noether's Theorem on Time Scales\thanks{Partially presented
at the 6th International ISAAC Congress, August 13 to August 18, 2007,
Middle East Technical University, Ankara, Turkey.}}

\author{Zbigniew Bartosiewicz${}^{\dag}$\\
        \texttt{bartos@pb.bialystok.pl}
        \and
        Delfim F.~M.~Torres${}^{\ddag}$\\
        \texttt{delfim@ua.pt}}

\date{${}^{\dag}$Faculty of Computer Science\\
      Bia\l ystok Technical University\\
      15-351 Bia\l ystok, Poland\\\medskip
      ${}^{\ddag}$Department of Mathematics\\
      University of Aveiro\\
      3810-193 Aveiro, Portugal}

\maketitle


\begin{abstract}
We show that for any variational symmetry of the problem of the
calculus of variations on time scales there exists a conserved
quantity along the respective Euler-Lagrange extremals.
\end{abstract}

\smallskip


\smallskip


\smallskip



\section{Introduction}

The calculus on time scales is a relatively new theory
that unifies and generalizes difference
and differential equations. The theory was initiated
by Stefan Hilger and is being applied to many different fields
in which dynamic processes can be described with discrete or continuous models
\cite{Agarwal,bookBohner,Hilger97}. The calculus of variations
on time scales was initiated in 2004 with the papers \cite{BohnerCV04,CVts}.

The calculus of variations and control theory are disciplines in which there
appears to be many opportunities for application of time scales
\cite{MR2211489,IFACPrague,Rui-Delfim}.
Here we make use of the Euler-Lagrange equations on time scales
\cite{BohnerCV04,comRui:TS:Lisboa07,CVts} to generalize
one of the most beautiful results of the
calculus of variations -- the celebrated Noether's theorem
\cite{GelfandFomin,MR0406752,Sarlet}.
Our Noether-type theorem (Theorem~\ref{theo:tnoe})
unifies and extends the previous formulations of Noether's
principle in the discrete-time and continuous domains
(\textrm{cf.} \cite{Torres03,Torres04} and references therein).
Moreover, it gives answer (Corollary~\ref{cor:NT:Disc}) to an open question
formulated in \cite[p.~216]{Logan} (see also \cite[Remark~12]{Torres03}):
\emph{how to obtain 'energy' integrals for discrete-time problems,
as done in the continuous calculus of variations?}


\section{Preliminaries on time scales}

We give here basic definitions and facts concerning the calculus
on time scales. More information can be found \textrm{e.g.}
in \cite{bookBohner}.

A {\it time scale} $\T$ is an arbitrary nonempty closed subset
of the set $\R$ of real numbers. It is a model of time. Besides standard
cases of $\mathbb{R}$ (continuous time) and $\mathbb{Z}$ (discrete time),
many different models are used. For each time scale $\mathbb{T}$
the following operators are used:
\begin{itemize}

\item the {\it forward jump operator} $\sigma:\T \rar \T$,
$\sigma(t):=\inf\{s \in \T:s>t\}$ for $t<\sup \T$ and
$\sigma(\sup\T)=\sup\T$ if $\sup\T<+\infty$;

\item the {\it backward jump operator} $\rho:\T \rar \T$,
$\rho(t):=\sup\{s \in \T:s<t\}$ for $t>\inf \T$ and
$\rho(\inf\T)=\inf\T$ if $\inf\T>-\infty$;

\item the {\it graininess function} $\mu:\T \rar [0,\infty)$,
$\mu(t):=\sigma(t)-t$.
\end{itemize}

\begin{example}  If $\T=\R$, then for any $t \in \R$,
$\sigma(t)=t=\rho(t)$ and $\mu(t) \equiv 0$.
If $\T=\Z$, then for every $t \in \Z$,
$\sigma(t)=t+1$, $\rho(t)=t-1$ and $\mu(t) \equiv 1$.
\end{example}

A point $t \in \mathbb{T}$ is called:
(i) {\it{ right-scattered}} if $\sigma(t)>t$,
(ii) {\it{right-dense}} if $\sigma(t)=t$,
(iii) {\it {left-scattered}} if $\rho(t)<t$,
(iv) {\it{left-dense}} if $\rho(t)=t$,
(v) {\it {isolated}} if it is both left-scattered
and right-scattered,
(vi) {\it{dense}} if it is both left-dense and right-dense.
If $\sup \T$ is finite and left-scattered,
we set $\mathbb{T}^\kappa :=\mathbb{T}\setminus \{\sup\T\}$.
Otherwise, $\mathbb{T}^\kappa :=\mathbb{T}$.

\begin{definition}
Let $f:\T \rar \R$ and $t \in \T^\kappa$. The {\it delta derivative}
of $f$ at $t$ is the real number $f^{\Delta}(t)$ with the property
that given any $\varepsilon$ there is a neighborhood
$U=(t-\delta,t+\delta) \cap \T$ of $t$ such that
\[|(f(\sigma(t))-f(s))-f^{\Delta}(t)(\sigma(t)-s)| \leq \varepsilon|\sigma(t)-s|\]
for all $s \in U$. We say that $f$ is {\it delta differentiable}
on $\T$ provided $f^{\Delta}(t)$ exists for all $t \in \T^\kappa$.
\end{definition}

We shall often denote $f^\Delta(t)$ by $\frac{\Delta}{\Delta t} f(t)$
if $f$ is a composition of other functions. The delta-derivative
of a function $f:\T \rar \R^n$ is a (column) vector whose components are
delta-derivatives of the components of $f$. For $f:\T \rar X$,
where $X$ is an arbitrary set, we define $f^\sigma:=f\circ\sigma$.

\begin{remark}
If $\T=\R$, then $f:\R \rar \R$ is
delta differentiable at $t \in \R$ if and only if $f$ is differentiable in
the ordinary sense at $t$. Then, $f^{\Delta}(t)=f'(t)$.
If $\T=\Z$, then $f:\Z \rar \R$ is always delta differentiable
at every $t \in \Z$ with $f^{\Delta}(t)=f(t+1)-f(t)$.
\end{remark}

\begin{definition}
A function $f:\mathbb{T} \to \mathbb{R}$ is called {\emph{rd-continuous}}
if it is continuous at the right-dense points in $\mathbb{T}$
and its left-sided limits exist at all left-dense points in $\mathbb{T}$.
A function $f:\mathbb{T} \to \mathbb{R}^n$ is
{\emph{rd-continuous}} if all its components
are rd-continuous.
\end{definition}

The set of all rd-continuous functions is denoted
by \(\mathcal{C}_{rd}\). Similarly, $\mathcal{C}^1_{rd}$
will denote the set of functions from $\mathcal{C}_{rd}$
whose delta derivative belongs to $\mathcal{C}_{rd}$.

A function $F:\mathbb{T}\to \mathbb{R}$ is called an
\emph{antiderivative of $f:\mathbb{T}\to \mathbb{R}$} if it
satisfies \( F^\Delta (t)=f(t), \) for all $t\in \mathbb{T}^\kappa$.
Then, the \emph{indefinite integral} of $f$ is defined by
\(\int f(t)\Delta t=F(t)+C,\) where $C$ is an arbitrary constant.
The \emph{definite integral} of $f$ is defined by
\(\int\limits_{r}^{s} f(t)\Delta t=F(s)-F(r),\)
for all $s,t\in \mathbb{T}$. It is known that every
rd-continuous function has an antiderivative.

\begin{example}
If $\mathbb{T}=\mathbb{R}$, then \[\int\limits_{a}^{b}  f(t) \Delta
t=\int\limits_{a}^{b}  f(t) d t,\] where the integral on the
right hand side is the usual Riemann integral.
If $\mathbb{T}=h\mathbb{Z}$, where $h>0$, then
\[\int\limits_{a}^{b}  f(t) \Delta
t=\sum\limits_{k=\frac{a}{h}}^{\frac{b}{h}-1}  h\cdot f(kh),\] for $a<b$.
\end{example}

We shall need the following properties
of delta derivatives and integrals:
\begin{gather}
(fg)^\Delta=f^\Delta g^\sigma + fg^\Delta \, , \\
f^\sigma=f+\mu f^\Delta \, , \\
\int_{a}^{b} f(\alpha(t))\alpha^\Delta(t)\Delta t = \int_{\alpha(a)}^{\alpha(b)} f(\bar{t})\bar{\Delta}\bar{t} \, ,
\end{gather}
where $\alpha : [a,b]\cap \mathbb{T}\rar\R$ is an increasing
$\mathrm{C}^1_{rd}$ function and its image is a new
time scale (so a new symbol $\bar{\Delta}\bar{t}$
is used in the second integral).


\section{Main results}

There exist several different ways to prove
the classical theorem \cite{MR0406752} of Emmy Noether.
Three different proofs of the classical
Noether's theorem  are reviewed
in \cite[Chap.~1]{PhD-thesisGastao}, a fourth one
can be found in \cite{GelfandFomin}. In this
section we follow one of those proofs,
which is based on a technique of time-reparameterization \cite{Jost:1998}.
While this approach is not so popular for proving the classical Noether's theorem and
many of its extensions (see \textrm{e.g.} \cite{GelfandFomin,Torres04,Brunt}),
it has already shown to be very effective in two generalizations
of the classical result: it has been used in \cite{Torres02} in order
to get a more general Noether's theorem in the optimal control setting;
it has been used in \cite{tncdf} to deal with
problems of the calculus of variations with fractional derivatives
in the Riemann-Liouville sense. Here we use
this technique to prove a Noether-type theorem for problems
of the calculus of variations on time scales.

We consider the fundamental problem of the calculus of variations
on time scales as defined by Bohner \cite{BohnerCV04} (see also
\cite{Ahlbrandt,economics,CVts}):
\begin{equation}
\label{Pt} I[q(\cdot)] = \int_{a}^{b}
L\left(t,q^{\sigma}(t),q^{\Delta}(t)\right)\Delta t \longrightarrow \min ,
\end{equation}
under given boundary conditions $q(a)=q_{a}$, $q(b)=q_{b}$,
where $q^{\sigma}(t)=(q\circ \sigma)(t)$, $q^{\Delta}(t)$
is the \emph{delta derivative}, $t\in \mathbb{T}$,
and the Lagrangian $L :\mathbb{R} \times \mathbb{R}^{n}
\times \mathbb{R}^{n} \rightarrow \mathbb{R}$
is a $\mathrm{C}^1$ function with respect to its
arguments. By $\partial_{i} L$ we will denote the partial
derivative of $L$ with respect to the $i$-th variable, $i = 1,2,3$.
Admissible functions $q(\cdot)$ are assumed to be $\mathrm{C}^1_{rd}$.

\begin{theorem}[See \cite{BohnerCV04}]
\label{Thm:ELts} If $q(\cdot)$ is a minimizer of problem \eqref{Pt},
then $q(\cdot)$ satisfies the following \emph{Euler-Lagrange equation}:
\begin{equation}
\label{eq:el} \frac{\Delta}{\Delta t}\partial_{3} L\left(t,q^\sigma(t),{q}^\Delta(t)\right) =
\partial_{2} L\left(t,q^\sigma(t),{q}^\Delta(t)\right) \, .
\end{equation}
\end{theorem}

\begin{definition}[invariance without transforming the time]
\label{def:inv} Let $U$ be a set of $\mathrm{C}^1_{rd}$ functions $q: [a,b] \rightarrow \mathbb{R}^n$. The
functional $I$ is said to be invariant on $U$ under a one-parameter family of state transformations
\begin{equation}
\label{eq:tinf} \bar{q}= q + \varepsilon\xi(t,q) + o(\varepsilon)
\end{equation}
if, and only if,
\begin{equation}
\label{eq:inv} \int_{t_{a}}^{t_{b}} L\left(t,q^\sigma(t),q^{\Delta}(t)\right)\Delta t =
\int_{t_{a}}^{t_{b}}L\left(t,\bar{q}^\sigma(t),\bar{q}^{\Delta}(t)\right)\Delta t
\end{equation}
for any subinterval $[{t_{a}},{t_{b}}] \subseteq [a,b]$ with ${t_{a}},{t_{b}}\in\mathbb{T}$, for any
$\varepsilon$ and for any $q\in U$, where $\bar{q}(t)= q(t) + \varepsilon\xi(t,q(t)) + o(\varepsilon)$.
\end{definition}

\begin{definition}[conservation law]
\label{def:leicon} Quantity $C(t,q,q^\sigma,{q}^\Delta)$ is said to be a \emph{conservation law} for functional
$I$ on $U$ if, and only if, $\frac{\Delta}{\Delta t}C(t,q(t),q^\sigma(t),{q}^\Delta(t))=0$ along all $q\in U$
that satisfy the Euler-Lagrange equation \eqref{eq:el}.
\end{definition}

\begin{theorem}[necessary condition of invariance]
\label{theo:cnsi} If functional $I$ is invariant on $U$ under transformations \eqref{eq:tinf}, then
\begin{equation}
\label{eq:cnsi}
\partial_{2} L\left(t,q^\sigma(t),{q}^\Delta(t)\right)\cdot\xi^\sigma(t,q(t))
+\partial_{3} L\left(t,q^\sigma(t),{q}^\Delta(t)\right)\cdot{\xi}^\Delta(t,q(t)) = 0
\end{equation}
for all $t\in[a,b]$ and all $q\in U$, where $\xi^\sigma(t,q(t))=\xi(\sigma(t)),q(\sigma(t)))$ and
${\xi}^\Delta(t,q(t))=\frac{\Delta}{\Delta t}{\xi}(t,q(t))$.
\end{theorem}

\begin{proof}
Having in mind that condition \eqref{eq:inv} is valid for any subinterval $[{t_{a}},{t_{b}}] \subseteq [a,b]$,
we can rid off the integral signs in \eqref{eq:inv}: equation \eqref{eq:inv} is equivalent to
\begin{equation}
\label{eq:inv1} L\left(t,q^\sigma(t),{q}^\Delta(t)\right) =
L(t,q^\sigma(t)+\varepsilon\xi^\sigma(t,q(t))+o(\varepsilon),{q}^\Delta(t) + \varepsilon{\xi}^\Delta(t,q(t))) \, .
\end{equation}
Differentiating both sides of equation \eqref{eq:inv1} with respect to $\varepsilon$, then setting
$\varepsilon=0$, we obtain equality \eqref{eq:cnsi}.
\end{proof}

\begin{theorem}[Noether's theorem without transforming time]
\label{theo:tnaut} If functional $I$ is invariant on $U$ under the one-parameter family of transformations
\eqref{eq:tinf}, then
\begin{equation}
\label{eq:tnau} C(t,q,q^\sigma,{q}^\Delta) = \partial_{3} L\left(t,q^\sigma,{q}^\Delta\right)\cdot\xi(t,q)
\end{equation}
is a conservation law.
\end{theorem}

\begin{proof}
Using the Euler-Lagrange equation \eqref{eq:el} and the necessary condition of invariance \eqref{eq:cnsi}, we
obtain:
\begin{equation*}
\begin{split}
\frac{\Delta}{\Delta t}&\left(\partial_{3} L\left(t,q^\sigma(t),{q}^\Delta(t)\right)
\cdot \xi(t,q(t))\right) \\
&= \frac{\Delta}{\Delta t}\partial_{3} L\left(t,q^\sigma(t),{q}^\Delta(t)\right) \cdot \xi^\sigma(t,q(t))
+ \partial_{3} L\left(t,q^\sigma(t),{q}^\Delta(t)\right) \cdot {\xi}^\Delta(t,q(t)) \\
&=  \partial_{2} L\left(t,q^\sigma(t),{q}^\Delta(t)\right) \cdot \xi^\sigma(t,q(t)) + \partial_{3}
L\left(t,q^\sigma(t),{q}^\Delta(t)\right) \cdot {\xi}^\Delta(t,q(t)) \\
&= 0 \, .
\end{split}
\end{equation*}
\end{proof}

\begin{remark}
In classical mechanics, for $\mathbb{T}=\mathbb{R}$, $\partial_{3} L\left(t,q^\sigma(t),{q}^\Delta(t)\right)$ is
interpreted as the generalized momentum.
\end{remark}

Let us consider now the one-parameter family of infinitesimal transformations
\begin{equation}
\label{eq:tinf2}
\begin{cases}
\bar{t} = T_\varepsilon(t,q) = t + \varepsilon\tau(t,q) + o(\varepsilon) \, ,\\
\bar{q} = Q_\varepsilon(t,q) = q + \varepsilon\xi(t,q) + o(\varepsilon) \, .\\
\end{cases}
\end{equation}
For a fixed $\varepsilon$ we will drop this index in $T_\varepsilon$ and $Q_\varepsilon$ and write $T$ and $Q$
instead. Let as before $U$ be a set of $\mathrm{C}^1_{rd}$ functions $q: [a,b] \rightarrow \mathbb{R}^n$.
We assume that for every $q\in U$ and every $\varepsilon$, the map $[a,b]\ni t \mapsto \alpha(t):=
T_\varepsilon(t,q(t))\in\mathbb{R}$ is an increasing $\mathrm{C}^1_{rd}$ function and its image is again a time
scale with the forward shift operator $\bar{\sigma}$ and the delta derivative $\bar{\Delta}$. Observe that the following holds:
\begin{equation}
\label{sigma}
\bar{\sigma}\circ\alpha = \alpha \circ \sigma .
\end{equation}
Let $\beta=\alpha^{-1}$. We set $\bar{q}(\bar{t}) := Q\left(\beta(\bar{t}),q(\beta(\bar{t}))\right)$.

\begin{definition}[invariance of $I$]
\label{def:inva} Functional $I$ is said to be invariant on $U$ under the family of  transformations
\eqref{eq:tinf2} if, and only if, for any subinterval $[{t_{a}},{t_{b}}] \subseteq [a,b]$, any $\varepsilon$ and
any $q\in U$
\begin{equation*}
\label{eq:inv2} \int_{t_{a}}^{t_{b}} L\left(t,q^\sigma(t),{q}^\Delta(t)\right)\Delta t =
\int_{T(t_a,q(t_a))}^{T(t_b,q(t_b))}
L\left(\bar{t},\bar{q}^{\bar{\sigma}}(\bar{t}),{\bar{q}}^{\bar{\Delta}}(\bar{t})\right) \bar{\Delta}\bar{t}.
\end{equation*}
\end{definition}

\begin{remark}
Observe that in Definition~\ref{def:inva} we change time.
Thus, we consider the functional $I$ on many
different time scales, depending on $\varepsilon$ and $q(\cdot)$.
This is the reason for assuming that the Lagrangian $L$
is defined for all $t\in\R$ and not just $t$ from the initial time scale $\T$.
\end{remark}

\begin{theorem}[Noether's theorem]
\label{theo:tnoe} If functional $I$ is invariant on $U$, in the sense of Definition~\ref{def:inva}, then
\begin{multline}
\label{eq:TeoNet} C(t,q,q^\sigma,{q}^\Delta) =
\partial_{3} L\left(t,q^\sigma,{q}^\Delta\right)\cdot\xi(t,q)\\
+ \left[ L(t,q^\sigma,{q}^\Delta) - \partial_{3} L\left(t,q^\sigma,{q}^\Delta\right) \cdot {q}^\Delta -
\partial_{1} L\left(t,q^\sigma,{q}^\Delta\right) \cdot \mu(t) \right]\cdot
\tau(t,q)
\end{multline}
is a conservation law.
\end{theorem}

\begin{proof}
We will show that invariance of $I$ under \eqref{eq:tinf2} (in the sense of Definition~\ref{def:inva}) is
equivalent to invariance of another functional $\tilde{I}$ in the sense of Definition~\ref{def:inv}.

Let $\tilde{L}(t;s,q;r,v):=L(s-\mu(t)r,q,\frac{v}{r})\cdot r$ for
$q,v\in\mathbb{R}^{n}$, $t\in [a,b]$ and
$s,r\in\mathbb{R}$, $r\neq 0$. $\tilde{L}$ is a new Lagrangian with the state variable
$(s,q)\in\mathbb{R}^{n+1}$. Observe that for $s(t)=t$ and any $q:[a,b]\rightarrow\mathbb{R}^n$
\[ L(t,q^\sigma (t),q^\Delta (t))=
\tilde{L}(t;s^\sigma(t),q^\sigma (t);s^\Delta(t),q^\Delta(t))
\]
so for the functional \[\tilde{I}[s(\cdot),q(\cdot)]:= \int_a^b
\tilde{L}(t;s^\sigma(t),q^\sigma(t);s^\Delta(t),q^\Delta(t))\Delta t\] we get
$I[q(\cdot)]=\tilde{I}[s(\cdot),q(\cdot)]$ whenever $s(t)=t$.

Consider the group of transformations $(T_\varepsilon,Q_\varepsilon)$ given by \eqref{eq:tinf2} and let $q\in
U$. From the invariance of $I$, for $s(t)=t$, we get
\begin{equation*}
\begin{split}
\tilde{I}[s(\cdot),q(\cdot)]&=I\left[q(\cdot)\right]
= \int_a^b L\left(t,q^\sigma(t),{q}^\Delta(t)\right)\Delta t \\
&= \int_{\alpha(a)}^{\alpha(b)} L\left(\bar{t},(\bar{q}\circ
\bar{\sigma})(\bar{t}),\bar{q}^{\bar{\Delta}}(\bar{t})\right)\bar{\Delta}\bar{t} \\
&= \int_{a}^{b} L\left(\alpha(t),(\bar{q}\circ \bar{\sigma}\circ\alpha)(t),
\bar{q}^{\bar{\Delta}}(\alpha(t))\right)\alpha^\Delta(t)\Delta{t} \\
&= \int_{a}^{b} L\left(\alpha^\sigma(t)-\mu(t)\alpha^\Delta(t),(\bar{q}\circ\alpha)^\sigma(t),
\frac{(\bar{q}\circ\alpha)^{\Delta}(t)}{\alpha^{\Delta}(t)}\right)\alpha^\Delta(t)\Delta{t} \\
&= \int_{a}^{b} \tilde{L}\left(t;\alpha^\sigma(t),(\bar{q}\circ\alpha)^\sigma(t);
\alpha^\Delta(t),(\bar{q}\circ\alpha)^\Delta(t))\right)\Delta t \\
&= \tilde{I}[\alpha(\cdot),(\bar{q}\circ\alpha)(\cdot)].
\end{split}
\end{equation*}

Observe that for $s(t)=t$
\[ (\alpha(t),(\bar{q}\circ\alpha)(t))= (T_\varepsilon(t,q(t)),
Q_\varepsilon(t,q(t)))= (T_\varepsilon(s(t),q(t)), Q_\varepsilon(s(t),q(t))). \]

 This means that $\tilde{I}$ is invariant on
$\tilde{U}=\{ (s,q)\ | \ s(t)=t, \, q\in U \}$ under the group of state transformations
\[    (\bar{s},\bar{q})= (T_\varepsilon(s,q),Q_\varepsilon(s,q)) \]
in the sense of Definition~\ref{def:inv}.
Applying Theorem~\ref{theo:tnaut}, we obtain that for $s(t)=t$
\begin{multline} \label{eq:dtn}
C\left(t,s,q,s^\sigma,q^\sigma,s^\Delta,q^\Delta\right)=\partial_{5}
\tilde{L}(t;s^\sigma,q^\sigma;s^\Delta,q^\Delta)\cdot\xi(s,q)
\\ + \partial_{4}\tilde{L}(t;s^\sigma,q^\sigma;s^\Delta,q^\Delta)
\cdot\tau(s,q)
\end{multline}
is a conservation law. Since
\begin{equation*}
\partial_{5} \bar{L}(t;s^\sigma,q^\sigma;s^\Delta,q^\Delta) =
\partial_{3} L\left(s^\sigma-\mu(t)s^\Delta,q^\sigma,\frac{q^\Delta}{s^\Delta}\right)
\end{equation*}
and
\begin{multline*}
\partial_{4} \bar{L}(t;s^\sigma,q^\sigma;s^\Delta,q^\Delta) =
-\partial_{1} L\left(s^\sigma-\mu(t)s^\Delta,q^\sigma,\frac{q^\Delta}{s^\Delta}\right)
\cdot\mu(t)\cdot s^\Delta\\
 -\partial_{3}
L\left(s^\sigma-\mu(t)s^\Delta,q^\sigma,\frac{q^\Delta}{s^\Delta}\right)
\cdot\frac{q^\Delta}{s^\Delta}+L\left(s^\sigma-\mu(t)s^\Delta,q^\sigma,\frac{q^\Delta}{s^\Delta}\right) \, ,
\end{multline*}
for $s(t)=t$ we get
\begin{equation}
\label{eq:dtn1}
\partial_{5} \bar{L}(t;s^\sigma,q^\sigma;s^\Delta,q^\Delta) =
\partial_{3} L\left(t,q^\sigma,q^\Delta\right)
\end{equation}
and
\begin{multline}\label{eq:dtn2}
\partial_{4} \bar{L}(t;s^\sigma,q^\sigma;s^\Delta,q^\Delta) \\
= L(t,q^\sigma,q^\Delta)-\partial_{3} L\left(t,q^\sigma,q^\Delta\right) \cdot q^\Delta-\partial_{1}
L\left(t,q^\sigma,q^\Delta\right) \cdot\mu(t).
\end{multline}
Substituting \eqref{eq:dtn1} and \eqref{eq:dtn2} into \eqref{eq:dtn} we arrive to the intended conclusion
\eqref{eq:TeoNet}.
\end{proof}

Assuming that $\mathbb{T}=\mathbb{R}$
the formula \eqref{eq:TeoNet} simplifies due to the fact that $\mu\equiv 0$,
and we obtain the classical Noether's theorem:

\begin{corollary}
Let $\mathbb{T}=\mathbb{R}$. If functional $I$ is
invariant on $U$, in the sense of Definition~\ref{def:inva}, then
\begin{equation*}
C(t,q,{q}') =
\partial_{3} L\left(t,q,{q}'\right)\cdot\xi(t,q)
+ \left[ L(t,q,{q}') - \partial_{3} L\left(t,q,{q}'\right) \cdot {q}' \right] \cdot \tau(t,q)
\end{equation*}
is a conservation law.
\end{corollary}

\begin{remark}
In classical mechanics, the term
$L(t,q,{q}') - \partial_{3} L\left(t,q,{q}'\right) \cdot {q}'$
is interpreted as the energy.
\end{remark}

For the discrete-time case ($\mathbb{T}=\mathbb{Z}$),
we obtain a new version of Noether's theorem which generalizes
the result in \cite{Torres03}:

\begin{corollary}
\label{cor:NT:Disc}
Let $\mathbb{T}=\mathbb{Z}$. If functional $I$ is invariant on $U$,
in the sense of Definition~\ref{def:inva}, then
\begin{multline*}
C(t,q,q^+,{\Delta q}) =
\partial_{3} L\left(t,q^+,{\Delta q}\right)\cdot\xi(t,q) \\
+ [ L(t,q^+,{\Delta q}) - \partial_{3} L\left(t,q^+,{\Delta q}\right) \cdot {\Delta q}
 -\partial_{1} L\left(t,q^+,{\Delta q}\right) ]\cdot
\tau(t,q)
\end{multline*}
is a conservation law, where $q^+(t)=q(t+1)$ and $\Delta q=q^+-q$.
\end{corollary}

We finish with an example of a conservation law
on a discrete but nonhomogeneous time scale
(graininess is not constant).

\begin{example}
Let $\mathbb{T}=\{ 2^n : n\in\mathbb{N}\cup\{0\} \}$ and
\[ L(t,q^\sigma, q^\Delta)=\frac{(q^\sigma)^2}{t} + t(q^\Delta)^2 \]
for $q\in\mathbb{R}$. It can be shown that the functional $I$
is invariant under the family of transformations:
\[ \bar{t}=te^\varepsilon=t+t\varepsilon + o(\varepsilon),\ \bar{q}=q. \]
Then, Noether's theorem generates the following conservation law:
\[ C(t,q^\sigma,q^\Delta)=2\left[\frac{(q^\sigma)^2}{t} - t(q^\Delta)^2\right]\cdot t. \]
\end{example}


\section*{Acknowledgments}

Zbigniew Bartosiewicz was supported by Bia\l ystok Technical University 
grant W/WI/1/07; Delfim F. M. Torres by the R\&D unit 
``Centre for Research in Optimization and Control'' (CEOC) 
of the University of Aveiro, cofinanced
by the European Community Fund FEDER/POCI 2010.



\end{document}